\documentclass[12pt]{amsart}
\usepackage{amsthm}
\usepackage{amsmath}
\usepackage{amssymb}
\usepackage{comment}
\usepackage{enumitem}
\usepackage{amsfonts}
\usepackage{mathtools}
\usepackage{setspace}
\usepackage[dvipsnames]{xcolor}
\usepackage[hidelinks]{hyperref}
\usepackage{dynkin-diagrams}
\usepackage{tikz}
\usetikzlibrary{decorations.markings,arrows.meta}
\hypersetup{
    colorlinks=true,
    linkcolor=blue,
    urlcolor=red,
    citecolor=Green
    }
\usepackage{subfigure}

\tikzset{
  midarrow/.style={
    postaction={decorate},
    decoration={markings, mark=at position 0.65 with {\arrow{Stealth[length=#1]}}}
  }
}

\usepackage[utf8]{inputenc}

\advance\paperheight1.00cm
\advance\textheight1.00cm
\advance\vsize1.00cm
\advance\topskip-0.5cm
\advance\voffset-0.5cm
\usepackage{tikz,tikz-cd}
\usepackage{todonotes}
\DeclareTextFontCommand{\emph}{\color{blue}\em}
\DeclareMathOperator{\ASM}{ASM}
\DeclareMathOperator{\CSM}{CSM}
\DeclareMathOperator{\MT}{MT}
\DeclareMathOperator{\rk}{rk}
\newcommand{\N}{\mathbb{N}}

\newcommand{\precdot}{\mathrel{\ooalign{$\prec$\cr
  \hidewidth\raise0.001ex\hbox{$\cdot\mkern0.6mu$}\cr}}}

\newtheorem{theorem}{Theorem}[section]
\newtheorem{proposition}[theorem]{Proposition}

\newtheorem{lemma}[theorem]{Lemma}

\newtheorem{question}[theorem]{Question}

\theoremstyle{definition}

\newtheorem{definition}[theorem]{Definition}

\theoremstyle{remark}

\title[MacNeille completions of parabolic quotients]{MacNeille completions of parabolic quotients}
\author{Yibo Gao}
\address{Beijing International Center for Mathematical Research, Peking University, Beijing 100871, China}
\email{gaoyibo@bicmr.pku.edu.cn}
\author{Hanlin Xu}
\address{School of Mathematical Sciences, Peking University, Beijing 100871, China}
\email{2401110030@stu.pku.edu.cn}
\thanks{Y.G. was partially supported by NSFC Grant No. 12471309} 
\date{\today}

\begin{document}
\pagestyle{plain}
\begin{abstract}
Alternating sign matrices (ASMs) arise as the Dedekind–MacNeille completion of the Bruhat order on the symmetric group. They enjoy fruitful combinatorial and geometric properties, with a particularly rich history on enumerations and bijections. In this paper, we explicitly describe the Dedekind-MacNeille completion of the Bruhat order on any parabolic quotients of the symmetric group. It is naturally a subposet of the alternating sign matrices, with different lattice operations. Moreover, we demonstrate the relations between the meet and join operations in this lattice with taking unions and intersection of the corresponding ASM varieties, respectively. Finally, we conclude with a more detailed discussion of special cases.
\end{abstract}
\maketitle

\section{Introduction}\label{sec:intro}
An \emph{alternating sign matrix (ASM)} is a square matrix with entries in $\{0,1,-1\}$ such that
\begin{itemize}
\item the nonzero entries in each and column alternate in sign, and
\item each row and column sums up to $1$.
\end{itemize}
Let $\ASM(n)$ denote the set of $n\times n$ alternating sign matrices. There is a rich history around enumerating the ASMs. The famous conjecture of Mills-Robbins-Rumsey \cite{MRR} provides a beautiful product formula \[|\ASM(n)|=\prod_{j=0}^{n-1}\frac{(3j+1)!}{(n+j)!},\]
which was first proved by Zeilberger \cite{Zei96} with technical arguments, and then by Kuperberg \cite{Kup96} with six-vertex models. There is also an explicit yet complicated bijective proof given by Fischer and Konvalinka \cite{FischerKonvalinka}. Besides, alternating sign matrices have seen strong connections with plane partitions \cite{Doran,BehrendDiFrancescoZinnJustin}, polytopes \cite{StrikerASMpolytope,MeszarosFlowASM} and multiple aspects around Schubert calculus \cite{WeigandtBPD,HuangPDTSSCPP,EKW25}.

Alternating sign matrices also arise as the Dedekind-MacNeille completion, also abbreviated as the MacNeille completion, of the (strong) Bruhat order of the symmetric group $S_n$ \cite{LS96}. It is then natural to ask for the MacNeille completion of the Bruhat order on any parabolic quotients $W^I:=W/W_I$, given by the Bruhat decomposition on the generalized partial flag variety $G/P_I$. Stembridge classified all parabolic quotients that are lattices \cite{StembridgeFullyCommutative}, but beyond that, little is known.

In this paper, we present a complete answer for arbitrary parabolic quotients in type $A$.
\begin{theorem}\label{thm:main}
Let $I\subset[n-1]$ be a subset. The Dedekind-MacNeille completion of the parabolic Bruhat order $S_n^I$ can be realized as
\[\ASM^I(n):=\{A\in\ASM(n)\:|\:r_A(i,j)\in\{r_A(i{-}1,j){+}1,r_A(i{+}1,j)\}\text{ for all }i\in I,j\in[n]\},\]
with the partial order induced from that of $\ASM(n)$.
\end{theorem}
See Figure~\ref{fig:MCn=4} for an example. In the description in Theorem~\ref{thm:main}, we do allow $r_A(i{-}1,j)+1$ and $r_A(i{+}1,j)$ to be equal, in which case $r_A(i,j)$ has to take on this value. In other words, the tuple $(r_A(i{-}1,j),r_A(i,j),r_A(i{+}1,j))$ is not allowed to be $(k,k,k{+}1)$ for any $i\in I$, $j\in[n]$ and some value $k\in\N$. In the language of \cite{EKW25}, this condition is saying that $A$ has no descent in rows of $I$. 

We remark that although $\ASM^I(n)$ is a subposet of $\ASM(n)$, the lattice operations do not fully agree. With our convention, taking join in $\ASM^I(n)$ agrees with that in $\ASM(n)$ while the meet does not (Theorem~\ref{thm:meet-operation}). One may take the maximal representatives to realize $S_n^I$ as a different subposet of $S_n$, in which case the situation is reversed.

Since our treatment for ASMs will be primarily based on their rank matrices, it is natural to discuss connections with ASM varieties. The following result generalizes Proposition 2.2 and Proposition 2.3 of \cite{EKW25}, with necessary notations provided in Section~\ref{sub:prelim-ASM}.
\begin{proposition}\label{prop:variety}
Let $A_1,\ldots,A_k\in\ASM^I(n)$.
\begin{enumerate}
\item Let $A=A_1\vee\cdots\vee A_k\in\ASM^I(n)$. Then $I_A=I_{A_1}+\cdots+I_{A_k}$ and $X_A=X_{A_1}\cap\cdots\cap X_{A_k}$.
\item Let $A=A_1\wedge_{I}\cdots\wedge_I A_k\in\ASM^I(n)$. If $X_{A_1}\cup\cdots\cup X_{A_k}$ is an $\ASM$ variety, then $I_A=I_{A_1}\cap\cdots\cap I_{A_k}$ and $X_A=X_{A_1}\cup\cdots\cup X_{A_k}$.
\end{enumerate}
\end{proposition}
Here, we use $\vee$ for the join operation in $\ASM^I(n)$ for arbitrary $I$, as they all agree with the join operation in $\ASM(n)$, and we use $\wedge_I$ to denote the meet operation in $\ASM^I(n)$, whose exact calculation is shown in Theorem~\ref{thm:meet-operation}.

There are some notable special cases: when $I=[n-1]\setminus\{k\}$ is maximal, $S_n^I$ is the Young's lattice under the rectangle $k\times(n-k)$, and we expect $\ASM^I(n)=S_n^I$ (Proposition~\ref{prop:I-maximal}); when $I$ is the interval $\{t,t+1,\cdots,n-1\}$, we obtain a clean description for $\ASM^I(n)$:
\begin{proposition}\label{prop:J=1...n-t}
For $I=\{t,t+1,\cdots,n-1\}\subset [n-1]$, $\ASM^I(n)$ is in bijection with the set of monotone triangles of size $(t-1)$ with entries less than or equal to $n$.
\end{proposition}

This paper is organized as follows. In Section~\ref{sec:prelim}, we introduce the necessary preliminary material. In Section~\ref{sec:proof}, we prove our main theorem, Theorem~\ref{thm:main}, discuss properties of the lattice $\ASM^I(n)$ including an explicit description of the lattice operations (Theorem~\ref{thm:meet-operation}) and also prove Proposition~\ref{prop:variety}. In Section~\ref{sec:further}, we end with special cases, enumerative results and further questions.

\begin{figure}
\centering
\begin{tikzpicture}[scale=0.6]
\def\a{0.8};
\def\b{0.4};
\def\c{1.5};
\def\ts{0.6};
\def\h{1.5};
\def\r{0.2};
\newcommand\Rec[3]{
\node[align=center] at (#1,#2) {#3};
\draw(#1-\a,#2-\b+\r)--(#1-\a,#2+\b-\r);
\draw(#1-\a+\r,#2+\b)--(#1+\a-\r,#2+\b);
\draw(#1+\a,#2+\b-\r)--(#1+\a,#2-\b+\r);
\draw(#1+\a-\r,#2-\b)--(#1-\a+\r,#2-\b);
\draw (#1-\a,#2-\b+\r) arc (180:270:\r);
\draw (#1+\a-\r,#2-\b) arc (270:360:\r);
\draw (#1+\a,#2+\b-\r) arc (0:90:\r);
\draw (#1-\a+\r,#2+\b) arc (90:180:\r);
}
\newcommand\RecBig[2]{
\draw(#1-\c,#2-\c+\r)--(#1-\c,#2+\c-\r);
\draw(#1-\c+\r,#2+\c)--(#1+\c-\r,#2+\c);
\draw(#1+\c,#2+\c-\r)--(#1+\c,#2-\c+\r);
\draw(#1+\c-\r,#2-\c)--(#1-\c+\r,#2-\c);
\draw (#1-\c,#2-\c+\r) arc (180:270:\r);
\draw (#1+\c-\r,#2-\c) arc (270:360:\r);
\draw (#1+\c,#2+\c-\r) arc (0:90:\r);
\draw (#1-\c+\r,#2+\c) arc (90:180:\r);
}
\newcommand\fournumbers[6]{
\node[align=center] at (#1-1.5*\ts,#2) {#3};
\node[align=center] at (#1-0.5*\ts,#2) {#4};
\node[align=center] at (#1+0.5*\ts,#2) {#5};
\node[align=center] at (#1+1.5*\ts,#2) {#6};
}
\RecBig{-2*\c}{0};
\fournumbers{-2*\c}{1.5*\ts}{0}{$+$}{0}{0};
\fournumbers{-2*\c}{0.5*\ts}{$+$}{$-$}{$+$}{0};
\fournumbers{-2*\c}{-0.5*\ts}{0}{0}{0}{$+$};
\fournumbers{-2*\c}{-1.5*\ts}{0}{$+$}{0}{0};
\RecBig{2*\c}{0};
\fournumbers{2*\c}{1.5*\ts}{0}{0}{$+$}{0};
\fournumbers{2*\c}{0.5*\ts}{$+$}{0}{0}{0};
\fournumbers{2*\c}{-0.5*\ts}{0}{$+$}{$-$}{$+$};
\fournumbers{2*\c}{-1.5*\ts}{0}{0}{$+$}{0};

\Rec{-4*\a}{-\c-\h}{$1342$};
\Rec{0*\a}{-\c-\h}{$2143$};
\Rec{4*\a}{-\c-\h}{$3124$};
\Rec{-2*\a}{-\c-2*\h}{$1243$};
\Rec{2*\a}{-\c-2*\h}{$2134$};
\Rec{0*\a}{-\c-3*\h}{$1234$};
\Rec{-4*\a}{\c+\h}{$2341$};
\Rec{0*\a}{\c+\h}{$3142$};
\Rec{4*\a}{\c+\h}{$4123$};
\Rec{-2*\a}{\c+2*\h}{$3241$};
\Rec{2*\a}{\c+2*\h}{$4132$};
\Rec{0*\a}{\c+3*\h}{$4231$};

\draw(-2*\c,-\c)--(-4*\a,-\c-\h+\b);
\draw(-2*\c,-\c)--(0*\a,-\c-\h+\b);
\draw(2*\c,-\c)--(4*\a,-\c-\h+\b);
\draw(2*\c,-\c)--(0*\a,-\c-\h+\b);
\draw(-4*\a,-\c-\h-\b)--(-2*\a,-\c-2*\h+\b);
\draw(0*\a,-\c-\h-\b)--(-2*\a,-\c-2*\h+\b);
\draw(0*\a,-\c-\h-\b)--(2*\a,-\c-2*\h+\b);
\draw(4*\a,-\c-\h-\b)--(2*\a,-\c-2*\h+\b);
\draw(-2*\a,-\c-2*\h-\b)--(0*\a,-\c-3*\h+\b);
\draw(2*\a,-\c-2*\h-\b)--(0*\a,-\c-3*\h+\b);

\draw(-2*\c,\c)--(-4*\a,\c+\h-\b);
\draw(-2*\c,\c)--(0*\a,\c+\h-\b);
\draw(2*\c,\c)--(4*\a,\c+\h-\b);
\draw(2*\c,\c)--(0*\a,\c+\h-\b);
\draw(-4*\a,\c+\h+\b)--(-2*\a,\c+2*\h-\b);
\draw(0*\a,\c+\h+\b)--(-2*\a,\c+2*\h-\b);
\draw(0*\a,\c+\h+\b)--(2*\a,\c+2*\h-\b);
\draw(4*\a,\c+\h+\b)--(2*\a,\c+2*\h-\b);
\draw(-2*\a,\c+2*\h+\b)--(0*\a,\c+3*\h-\b);
\draw(2*\a,\c+2*\h+\b)--(0*\a,\c+3*\h-\b);
\end{tikzpicture}
\caption{The MacNeille completion $\ASM^I(n)$ of $S_n^I$ for $n=4$ and $I=\{2\}$}
\label{fig:MCn=4}
\end{figure}
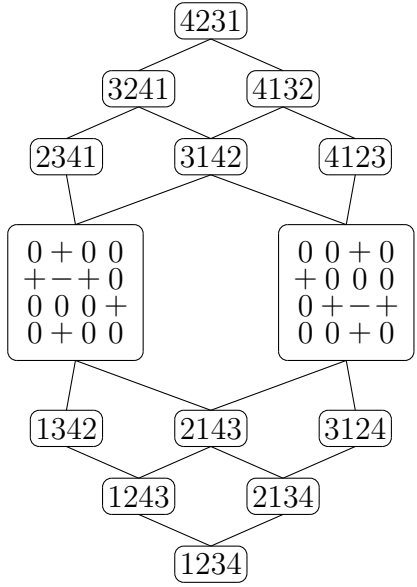

\section{Preliminaries}\label{sec:prelim}
\subsection{MacNeille completion of posets}\label{sub:MacNeille}
Let $P$ be a poset. The \emph{join} (resp. \emph{meet}) of $x,y\in P$, if it exists, is the unique least upper bound (resp. greatest lower bound), denoted $x\vee y$ (resp. $x\wedge y$). We can similarly define the join and meet of arbitrary subsets. The poset $P$ is called a \emph{lattice} if every pair of elements has a join and a meet, and is called a \emph{complete lattice} if any arbitrary subset has a join and a meet. In case of finite posets, these two notions conincide. An element $x\in P$ is called \emph{join-irreducible} (resp. \emph{meet-irreducible}) if it cannot be written as $s\vee t$ for $s,t<x$ (resp. $s\wedge t$ for $s,t>x$). The following is Proposition~3.3.1 in \cite{EC1}.

\begin{proposition}\label{lattice}
Let $P$ be a finite poset with maximum $\hat1$ (resp. minimum $\hat0$) such that every pair of elements has a meet (resp. join). Then $P$ is a lattice.
\end{proposition}

The \emph{Dedekind-MacNeille completion} of a poset $P$ is the smallest complete lattice which contains $P$ as an order embedding. For a lattice $L$ containing $P$, we say that $P$ is \emph{join-dense} (resp. \emph{meet-dense}) in $L$, if for each $x\in L$, $x=\bigvee_{y\in P,y\leq x}y$ (resp. $x=\bigwedge_{y\in P,y\geq x}y$). Here is a useful tool for checking MacNeille completion.
\begin{proposition}[\cite{BB67}]\label{prop:MacNeille0}
Let $P$ be a poset embedded in a lattice $L$, then $L$ is the MacNeille completion of $P$ if and only if $P$ is both meet-dense and join-dense in $L$.
\end{proposition}

\subsection{Permutations and alternating sign matrices}\label{sub:Bruhat}
Let $S_n$ be the symmetric group of permutations. We write a permutation $w\in S_n$ via its one-line notation $w(1)w(2)\cdots w(n)$. We also represent a permutation $w$ by its \emph{permutation matrix} $\dot w$, with $1$'s in the entries $(i,w(i))$ for $i\in[n]$, and $0$'s elsewhere. 
The group $S_n$ is generated by the set of \emph{simple reflections} $S=\{s_i:=(i\ i{+}1)\:|\:i\in[n-1]\}$. Denote by $T=\{t_{ij}:=(i\ j)\:|\:1\leq i<j\leq n\}$ the set of \emph{reflections}. For $w\in S_n$, its \emph{length} $\ell(w)$ is the smallest $\ell$ such that $w$ can be written as a product of simple reflections.

The \emph{Bruhat order} on $S_n$ is the partial order generated by $w<wt_{ij}$ if $\ell(w)<\ell(wt_{ij})$. There are many properties and characterizations for the Bruhat order, and we will use the one most compatible with alternating sign matrices. For $w\in S_n$, its \emph{rank matrix} is defined by $r_w(i,j):=|\{k\leq i\:|\: w(k)\leq j\}|$. For example, the permutation matrix and the rank matrix for $w=1342$ are \[\dot w=\begin{bmatrix}
1&0&0&0\\0&0&1&0\\0&0&0&1\\0&1&0&0
\end{bmatrix},\quad r_w=\begin{bmatrix}
1&1&1&1\\1&1&2&2\\1&1&2&3\\1&2&3&4
\end{bmatrix}.\]

The following result is very classical. See for example \cite{BB05}.
\begin{theorem}\label{thm:Rankcri}
For $u,w\in S_n$, $u\leq w$ if and only if $r_u(i,j)\geq r_w(i,j)$ for all $i,j\in[n]$.
\end{theorem}

For a subset $I\subset[n-1]$, the \emph{parabolic subgroup} $S_n(I)$ is generated by $\{s_i\:|\: i\in I\}$. We identify the \emph{parabolic quotient} $S_n/S_n(I)$ as the set of minimal representatives $S_n^I:=\{w\in S_n\:|\: w(i)<w(i+1)\text{ for all }i\in I\}$. They inherit the Bruhat order from $S_n$. Given $I\subset[n-1]$, every $w\in S_n$ admits a \emph{parabolic decomposition} $w=w^Iw_I$ that is length-additive such that $w^I\in S_n^I$ and $w_I\in S_n(I)$. 

For an alternating sign matrix $A\in\ASM(n)$, its \emph{rank matrix} is defined similarly as $r_A(i,j)=\sum_{p=1}^i\sum_{q=1}^j A_{p,q}$. The \emph{Bruhat order} on $\ASM(n)$ extends naturally by $A\leq B$ if $r_A(i,j)\geq r_B(i,j)$ for all $i,j\in[n]$. These matrices are also called the \emph{corner sum matrices}.
\begin{definition}\label{def:CSM}
An $n\times n$ \emph{corner sum matrix} $C=(c_{ij})_{i,j=1}^n$ is a square matrix with entries in $\N$ such that
\begin{itemize}
\item $c_{n,n}=n$,
\item $c_{i,j}-c_{i-1,j}\in\{0,1\}$ for all $i,j\in[n]$, and
\item $c_{i,j}-c_{i,j-1}\in\{0,1\}$ for all $i,j\in[n]$,
\end{itemize}
with the convention that $c_{i,0}=c_{0,j}=0$ for $i,j\in[n]$.
\end{definition}
Write $\CSM(n)$ for the set of $n\times n$ corner sum matrices. Taking the rank matrix provides a bijection from $\ASM(n)$ to $\CSM(n)$. The lattice structure on $\CSM(n)$ is given by entrywise operations. The meet and join are calculated by taking the entrywise maximum and minimum, respectively.

There are other combinatorial objects in bijection with $\ASM(n)$.
\begin{definition}
A \emph{monotone triangle} of size $n$ is a triangular array $(m_{i,j})_{1 \leq j \leq i \leq n}$ consisting of integers such that:
\begin{itemize}
\item $m_{i,j}\leq m_{i-1,j}\leq m_{i,j+1}$ for all $1\leq j<i\leq n$,
\item $m_{i,j}<m_{i,j+1}$ for all $1\leq j<i\leq n$.
\end{itemize}
\end{definition}
Given $A=(a_{ij})_{i,j=1}^n\in\ASM(n)$, we can construct a monotone triangle $(m_{i,j})_{1\leq j\leq i\leq n}$ where its $i^{th}$ row $m_{i,1}<m_{i,2}<\cdots<m_{i,i}$ consists of the column indices $c$ such that $a_{1,c}+a_{2,c}+\cdots+a_{i,c}=1$. This provides a bijection between $\ASM(n)$ and monotone triangles with bottom row $(m_{n,1},\ldots,m_{n,n})=(1,2,\ldots,n)$. An example is shown in Figure~\ref{fig:bijection-ASM-MT-example}.
\begin{figure}[h!]
\centering
\[\begin{pmatrix}
    0&1&0&0&0&0\\
    0&0&0&1&0&0\\
    1&0&0&-1&1&0\\
    0&0&1&0&0&0\\
    0&0&0&1&-1&1&\\
    0&0&0&0&1&0
\end{pmatrix}\qquad\begin{array}{c}
 2 \\
 2 \quad 4 \\
 1 \quad 2 \quad 5 \\
1 \quad 2 \quad 3 \quad 5 \\
1 \quad 2 \quad 3 \quad 4 \quad 6 \\
1 \quad 2 \quad 3 \quad 4 \quad 5 \quad 6\end{array}\]
\caption{Bijection between alternating sign matrices (left) and monotone triangles (right)}
\label{fig:bijection-ASM-MT-example}
\end{figure}

The join-irreducible elements in $S_n$ and in $\ASM(n)$ are precisely the bigrassmannian permutations \cite{LS96}. In this paper, we use the notation $P[a,b,c]\in S_n$ for them, where $a,b,c\in\N$, $a+b-n\leq c\leq a,b$, defined as
\[P[a,b,c](i)=\begin{cases}
    i, & i\leq c\text{ or }i> a+b-c, \\
    b-c+i, & c<i\leq a, \\
    -a+c+i, & a<i\leq a+b-c.
\end{cases}\]
Similarly, we also define the following permutations, for $(a,b,c)$ in the same range:
\[Q[a,b,c](i)=\begin{cases}
    n+1-i, & i\leq a-c\text{ or }i>n-b+c,\\
    a+b-c+1-i, & a-c<i\leq a,\\
    n+c+1-i, & a<c\leq n-b+c.
\end{cases}\]
These permutations are shown pictorially in Figure~\ref{fig:PQ-matrix}, where $\mathbf{I}_m$ is the $m\times m$ identity matrix, and $\mathbf{J}_m$ is the $m\times m$ matrix with $1$'s on the antidiagonal, and $0$'s elsewhere. 
\begin{figure}[h!]
\centering
\[
\begin{pmatrix}
\mathbf{I}_c & 0 & 0 & 0 \\
0 & 0 & \mathbf{I}_{a - c} & 0 \\
0 & \mathbf{I}_{b-c} & 0 & 0 \\
0 & 0 & 0 & \mathbf{I}_{n-a-b+c}
\end{pmatrix}\qquad
\begin{pmatrix}
0 & 0 & 0 &\mathbf{J}_{a-c}\\
0 & \mathbf{J}_{c} & 0 & 0 \\
0 & 0 & \mathbf{J}_{n-a-b+c} & 0 \\
\mathbf{J}_{b-c} & 0 & 0 & 0
\end{pmatrix}
\]
\caption{The join-irreducible elements $P[a,b,c]$ (left) and meet-irreducible elements $Q[a,b,c]$ (right)}
\label{fig:PQ-matrix}
\end{figure}
The following technical results can be found in \cite{LS96}.
\begin{proposition}\label{prop:joinirr-meetirr}
Let $a,b,c\in\N$ with $a+b-n\leq c\leq a,b$, then we have
\begin{enumerate}
\item $r_{P[a,b,c]}(a,b)=r_{Q[a,b,c]}(a,b)=c$.
\item For any $M\in\ASM(n)$ such that $r_M(a,b)=c$, $P[a,b,c]\leq M\leq Q[a,b,c]$.
\item $P[a,b,c]$ and $Q[a,b,c]$ are join-irreducible and meet-irreducible elements in $\ASM(n)$ respectively.
\end{enumerate}
\end{proposition}

\subsection{Matrix Schubert varieties}\label{sub:prelim-ASM}
Given $A \in \ASM(n)$, we can associate an \emph{ASM variety}
\[
X_A:=\{Z=(z_{i,j})_{i,j=1}^n\in\mathrm{Mat}_{n\times n}\mid\rk(Z_{[i],[j]})\leq r_A(i,j)\text{ for all }i,j \in [n]\}\]
where $Z_{I,J}$ denotes the submatrix with rows indexed by $I$ and columns indexed by $J$. Denote by $J_{i,j,k}$ the ideal generated by the size $k$-minors of $Z_{[i],[j]}$. Then the corresponding ideal of $X_A$ is $I_A=\sum_{i,j\in[n]}J_{i,j,r_A(i,j)+1}$ with $V(I_A)=X_A$.

The following theorem shows how ASM varieties behave under intersection and union.
\begin{theorem}[\cite{EKW25}]\label{thm:asmvar}
Let $A_1,\ldots,A_k\in\ASM(n)$.
\begin{enumerate}
\item Let $A=A_1\vee\cdots\vee A_k\in\ASM(n)$. Then $I_A=I_{A_1}+\cdots+I_{A_k}$ and $X_A=X_{A_1}\cap\cdots\cap X_{A_k}$.
\item Let $A=A_1\wedge\cdots\wedge A_k\in\ASM(n)$. If $X_{A_1}\cup\cdots\cup X_{A_k}$ is an $\ASM$ variety, then $I_A=I_{A_1}\cap\cdots\cap I_{A_k}$ and $X_A=X_{A_1}\cup\cdots\cup X_{A_k}$.
\end{enumerate}
\end{theorem}
\section{MacNeille completion of $S_n^I$}\label{sec:proof}
In this section, we prove the main theorem, that $\ASM^I(n)$ is the MacNeille completion of $S_n^I$. We first show that $\ASM^I(n)$ is a lattice, and then provide the exact formulas of the join and the meet operations in $\ASM^I(n)$. 

From now on, fix a subset $I\subset[n-1]$, and write it as a disjoint union of connected components in the type $A_{n-1}$ Dynkin diagram as $I=\cup_{1\leq t\leq k}\{u_t,u_t+1,\cdots,v_t\}$, where $u_t\leq v_t\in[n-1]$ and $v_t+1\leq u_{t+1}-1$ for $t\in[k]$.
\subsection{Properties of $\ASM^I$}
\begin{proposition}\label{ASMIlattice}
$\ASM^I(n)$ is a lattice.
\end{proposition}
\begin{proof}
The minimal element of $\ASM^I(n)$ is the identity matrix. 

For any $A,B\in\ASM^I(n)$, we will show that the join of $A$ and $B$ in $\ASM^I(n)$ exists. Let $C$ be the join of $A$ and $B$ in $\ASM(n)$. We claim that $C\in\ASM^I(n)$.

Arguing contradictorily that $C\notin\ASM^I(n)$, there exists $i\in I$ such that $r_C(i,j)=r_C(i-1,j)=r_C(i+1,j)-1$. Since $A\vee B = C$, we have $r_C(i,j)=\min\{r_A(i,j),r_B(i,j)\}$. Without loss of generality, assume $r_A(i,j)=r_C(i,j)$.  Since $r_C(i-1,j)=\min\{r_A(i-1,j),r_B(i-1,j)\}$, we have $r_A(i-1,j)\geq r_C(i-1,j)=r_C(i,j)=r_A(i,j)$. This means $r_A(i-1,j)=r_A(i,j)$. Similarly, we obtain $r_A(i,j)=r_A(i+1,j)-1$, contradicting the fact that $A\in\ASM^I(n)$.

Now that $C\in\ASM^I(n)$, it is the join of $A$ and $B$ in $\ASM^I(n)$. Using Proposition \ref{lattice}, we conclude that $\ASM^I(n)$ is a lattice.    
\end{proof}
To prove that $\ASM^I(n)$ is exactly the MacNeille completion of $S_n^I$, more properties need to be established. The following lemma is easy but useful for later calculations.
\begin{lemma}\label{lem:ineqt}
For $a,b,c,d\in\mathbb{R}$,
    \begin{align*}
        \min\{a-b,c-d\}\leq\min\{a,c\}&-\min\{b,d\}\leq\max\{a-b,c-d\},\\
        \min\{a-b,c-d\}\leq\max\{a,c\}&-\max\{b,d\}\leq\max\{a-b,c-d\}.
    \end{align*}
\end{lemma}
\begin{proof}
First, assume without loss of generality that $a \leq c$ . Then,
    \begin{align*}
        \min\{a,c\}-\min\{b,d\}&\geq a-b
        \geq\min\{a-b,c-d\}, \\
        \max\{a,c\}-\max\{b,d\}&\leq c-d
        \leq\max\{a-b,c-d\}.
    \end{align*}

    The other two inequalities can be proved in the same way by presetting the order relation between $b$ and $d$.
\end{proof}
The following lemma gives a more explicit description for elements in $\ASM^I(n)$.
\begin{lemma}\label{lem:ASMI}
Let $A\in\ASM(n)$. Then $A\in\ASM^I(n)$ if and only if for all $i\in\{u_t,u_t+1,\cdots,v_t\}$, 
\begin{equation}\label{eqt:rA}
r_A(i,j)=\min\{r_A(v_t+1,j),r_A(u_t-1,j)+(i-u_t+1)\}.
\end{equation}
\end{lemma}
\begin{proof}
If $A\in\ASM^I(n)$, then for $i\in \{u_t,u_t+1,\cdots,v_t\}$, we have $r_A(i,j)=r_A(i-1,j)+1$ or $r_A(i+1,j)=r_A(i,j)$. Consider the $\{0,1\}$-sequence $r_A(i,j)-r_A(i-1,j)$ for $i=u_t,\ldots,v_t+1$. By the definition of $\ASM^I(n)$ (Theorem~\ref{thm:main}), $1$ cannot follow any $0$'s in this sequence. In other words, there exists some $r\in\{u_t-1,u_t,u_t+1,\cdots,v_t,v_t+1\}$ such that 
\[r_A(i,j)=\begin{cases}
    r_A(i-1,j)+1, & \text{ for }u_t\leq i\leq r, \\
    r_A(i+1,j), & \text{ for }r\leq i\leq v_t.
\end{cases}\]
A straightforward calculation shows that 
\[r_A(i,j)=\begin{cases}
    r_A(u_t-1,j)+(i-u_t+1), & \text{ for }u_t\leq i\leq r, \\
    r_A(v_t+1,j), & \text{ for }r\leq i\leq v_t,
\end{cases}\]which is the same as Equation~\eqref{eqt:rA}.
    
For the converse, assume that Equation~\eqref{eqt:rA} holds. We need to show that for all $i\in I$, $r_A(i,j)=r_A(i-1,j)+1$ or $r_A(i+1,j)=r_A(i,j)$. Suppose that for some $i\in\{u_t,u_t+1,\cdots,v_t\}$, $r_A(i,j)=r_A(i-1,j)$. We will show that $r_A(i+1,j)=r_A(i,j)$. We first claim that $r_A(i,j)=r_A(v_t+1,j)$, since otherwise,
\begin{align*}
r_A(i,j)=&r_A(u_t-1,j)+(i-u_t+1)\\
>&\min\{r_A(v_t+1,j),r_A(u_t-1,j)+(i-1-u_t+1)\}\geq r_A(i-1,j)
\end{align*}
contradicting $r_A(i,j)=r_A(i-1,j)$. Therefore,
\begin{align*}
r_A(i+1,j)=&\min\{r_A(v_t+1,j),r_A(u_t-1,j)+(i+1-1-u_t+1)\}\\
\leq& r_A(v_t+1,j)=r_A(i,j).
\end{align*}
We obtain $r_A(i+1,j)=r_A(i,j)$ as desired.
\end{proof}
\begin{definition}\label{def:leqI}
    For $A,B\in\ASM(n)$, we say $A\leq_IB$ if and only if \[r_A(i,j)\geq r_B(i,j)\text{ for all }i\notin I\text{ and }j\in[n].\]
    If $A\geq_IB$ and $A\leq_IB$, we will write $A=_IB$.
\end{definition}
Note that this order is not a partial order on $\ASM(n)$, However, the next lemma shows it is a partial order on $\ASM^I(n)$. This means that to compare elements in $\ASM^I(n)$, we only need to compare entries with column index not in $I$.
\begin{lemma}\label{lem:leqI}
For $A,B\in\ASM(n)$, the followings are true.
\begin{enumerate}
\item If $A\in\ASM^I(n)$ and $B\in\ASM(n)$, then $A\leq B$ if and only if $A\leq_IB$.
\item If $A,B\in\ASM^I(n)$, then $A\leq B$ if and only if $A\leq_IB$.
\item If $A,B\in\ASM^I(n)$, then $A=B$ if and only if $A =_I B$.
\end{enumerate}
\end{lemma}
\begin{proof}
(2) is a corollary of (1), and (3) is a corollary of (2). So we only need to prove (1).

If $A\leq B$, then $A\leq_I B$ by definition. If $A\nleq B$, assume $r_A(i,j)< r_B(i,j)$ for some $i,j\in[n]$. If $i\notin I$, then $A\nleq_I B$. If $i\in I$, suppose that $i\in\{u_t,u_t+1,\cdots,v_t\}$. Then by Lemma \ref{lem:ASMI}, \[r_A(i,j)=\min\{r_A(v_t+1,j),r_A(u_t-1,j)+(i-u_t+1)\}.\]
If $r_A(i,j)=r_A(v_t+1,j)$, then \[r_A(v_t+1,j)= r_A(i,j)< r_B(i,j)\leq r_B(v_t+1,j),\] where $v_t+1\notin I$, and thus $A\nleq_I B$. If $r_A(i,j)=r_A(u_t-1,j)+(i-u_t+1)$, then\[r_A(u_t-1,j)=r_A(i,j)-(i-u_t+1)< r_B(i,j)-(i-u_t+1) \leq r_B(u_t-1,j),\] where $u_t-1\notin I$, and thus $A\nleq_I B$.
\end{proof}
Now we explicitly describe the meet and join operations in $\ASM^I(n)$.
\begin{theorem}\label{thm:meet-operation}
For $A,B\in\ASM^I(n)$, the join of $A$ and $B$ in $\ASM^I(n)$ agrees with that in $\ASM(n)$, while the meet of $A$ and $B$ in $\ASM^I(n)$, denoted by $C$, is given by
\begin{equation}\label{eqn:meet-expression}
r_C(i,j)=\begin{cases}
    \max\{r_A(i,j),r_B(i,j)\}, & \text{ for }i\in [n]\setminus I, \\
    \min\{r_C(v_t+1,j),r_C(u_t-1,j)+(i{-}u_t{+}1))\}, & \text{ for }u_t\leq i\leq v_t.
\end{cases}
\end{equation}
\end{theorem}
\begin{proof}
It is proved in Proposition \ref{ASMIlattice} that the join of $A$ and $B$, as elements in the lattice $\ASM(n)$, is also in $\ASM^I(n)$. Since $\ASM^I(n)$ is a subposet of $\ASM(n)$, the joins of $A$ and $B$ in these two lattices coincide.

For the core of this theorem, let $M$ be a matrix $M$ with entries defined as in Equation~\eqref{eqn:meet-expression}. By a direct application of Lemma~\ref{lem:ineqt}, we see that $M_{i+1,j}-M_{i,j}\in\{0,1\}$ and $M_{i,j+1}-M_{i,j}\in\{0,1\}$. Moreover, $M_{n,n}=n$ from construction. This means that $M$ is a corner sum matrix (Definition~\ref{def:CSM}), which bijects to an alternating sign matrix $C\in\ASM(n)$ with $r_C=M$. By Lemma~\ref{lem:ASMI}, $C\in\ASM^I(n)$. It remains to show that $C=A\wedge B$ in $\ASM^I(n)$.

First, $r_C(i,j)=\max\{r_A(i,j),r_B(i,j)\}$ for $i\notin I$ and $j\in[n]$, we have $C\leq_I A,B$ by Definition~\ref{def:leqI}. By Lemma \ref{lem:leqI}(2), $C\leq A,B$. Next, take any $X\in\ASM^I(n)$ with $X\leq A,B$, we have $r_X(i,j)\geq\max\{r_A(i,j),r_B(i,j)\}=r_C(i,j)$ for $i\notin I$ and $j\in[n]$. So $X\leq_I C$. Again by Lemma \ref{lem:leqI}(2), $X\leq C$ as desired.
\end{proof}

\subsection{MacNeille completion of $S_n^I$}
Recall that for $w\in S_n$, we have the parabolic decomposition $w=w^I w_I$, where $w_I$ is in the subgroup of $S_n$ generated by $\{s_i\:|\:i\in I\}$, and $w^I$ is the minimal coset representative of $w$ with respect to this subgroup.
\begin{lemma}\label{pISn}
For $w \in S_n$, $r_{w^I}(i,j)=r_w(i,j)$ for $i\notin I$ and $j\in[n]$.
\end{lemma}
\begin{proof}
We will in fact show that $r_{wv}(i,j)=r_{w}(i,j)$ for $i\notin I$, $j\in [n]$ and $v\in S_n(I)$. By induction on $\ell(v)$, we only need to prove it for $v=s_t$ where $t\in I$. 

For $i\notin I$, $s_t(k)\leq i$ if and only if $k\leq i$. For $i\notin I$ and $j\in[n]$,
\begin{align*}
r_{ws_t}(i,j)=&|\{k\:|\: k\leq i,\ ws_t(k)\leq j\}|=|\{k\:|\: s_t(k)\leq i,\ ws_t(k)\leq j\}|\\
=&|\{k\:|\: k\leq i,\ w(k)\leq j\}|=r_w(i,j).
\end{align*}
\end{proof}
For the main proof, we will need the following slightly stronger version of Proposition~\ref{prop:MacNeille0}.
\begin{proposition}\label{prop:MacNeille}
Let $P$ be a poset embedded in a lattice $L$. If for each element $x \in L$, there exist elements $y_1,y_2,\cdots,y_k$ and $z_1,z_2,\cdots,z_m$ in $P$ such that $x = \vee_{i=1}^ky_i=\wedge_{i=1}^mz_i$, then $L$ is the MacNeille completion of $P$.
\end{proposition}
\begin{proof}  
Since $x=\vee_{i=1}^k y_i$, we have $y_i \in \{y \in P\mid y\leq x\}$ for $i\in[k]$. Hence, $x=\vee_{i=1}^ky_i\leq\vee_{y\in P,y\leq x}y\leq x$ and thus $x=\vee_{y\in P,y\leq x}y$. The other equality $x=\wedge_{y\in P,z\geq x} z$ is analogous. This means that $P$ is both meet-dense and join-dense in $L$, and thus $L$ is the MacNeille completion of $P$ by Proposition~\ref{prop:MacNeille0}.
\end{proof}
We are now ready to prove the main theorem.
\begin{proof}[Proof of Theorem~\ref{thm:main}]
For convenience of notations, we write $r_A(i,j)$ as $a_{i,j}$ in this proof. Our strategy is to show that for any $A\in\ASM^I(n)$, \[A=\sideset{}{_{I}}\bigwedge_{i,j\in[n]}Q[i,j,a_{i,j}]^I\]
where $P[a,b,c]$, $Q[a,b,c]$ are the permutations defined in Section~\ref{sub:Bruhat}. Let $B\in\ASM^I(n)$ be the meet of these elements on the right hand side above. By Lemma~\ref{lem:leqI}(3), we only need to show that $A=_IB$. For $i\notin I$, using Proposition \ref{prop:joinirr-meetirr}(1), Lemma~\ref{pISn} and Theorem \ref{thm:meet-operation},\[a_{i,j}=r_{Q[i,j,a_{i,j}]}(i,j)=r_{Q[i,j,a_{i,j}]^I}(i,j)\leq\max_{u,v\in[n]}\{r_{Q[u,v,a_{u,v}]^I}(i,j)\}=r_B(i,j).\]
On the other hand, for $i\notin I$, using Proposition \ref{prop:joinirr-meetirr}(2), Lemma \ref{pISn} and Theorem \ref{thm:meet-operation}, \[a_{i,j}\geq\max_{u,v\in[n]}\{r_{Q[u,v,a_{u,v}]}(i,j)\}=\max_{u,v\in[n]}\{r_{Q[u,v,a_{u,v}]^I}(i,j)\}=r_B(i,j).\]
Therefore, $r_A(i,j)=r_B(i,j)$ for $i\notin I$ and $j\in [n]$, and thus $A=_IB$. In the same way, we can show that \[A =\bigvee_{i,j\in[n]} P[i,j,a_{i,j}]^I.\]
We conclude that $\ASM^I(n)$ is the MacNeille completion of $S_n^I$ by Proposition~\ref{prop:MacNeille}.
\end{proof}

\subsection{Relations to ASM varieties}
We work towards Proposition~\ref{prop:variety}. We first show that to determine $X_A$ for $A\in\ASM^I(n)$, one only needs those entries with row indices not in $I$.
\begin{lemma}\label{lem:asmvar}
Let $A\in\ASM^I(n)$, then\[X_A=\{Z\in\mathrm{Mat}_{n\times n}\mid\mathrm{rk}(Z_{[i],[j]})\leq r_A(i,j)\text{ for }i\notin I, j\in[n]\}.\]
\end{lemma}
\begin{proof}
Denote $X_A^I=\{Z\in\mathrm{Mat}_{n\times n}\mid\mathrm{rk}(Z_{[i],[j]})\leq r_A(i,j)\text{ for }i\notin I, j\in[n]\}$. By definition, $X_A^I\supset X_A$. So it suffices to show $X_A^I\subset X_A$. Let $M\in X_A^I$. We need to show that $\rk(M_{[i],[j]})\leq r_A(i,j)$ for $i\in I$ and $j\in[n]$. Suppose $i$ is in the range $\{u_t,u_t+1,\cdots,v_t\}$. For $0\leq i_1\leq i_2\leq n$, $\rk(M_{[i_1],[j]})\leq\rk(M_{[i_2],[j]})$ and $\rk(M_{[i_2],[j]})-\rk(M_{[i_1],[j]})\leq \rk(M_{[i_2]\backslash[i_1],[j]})\leq i_2-i_1$. Therefore, \[\rk(M_{[i],[j]})\leq\min\{\rk(M_{[v_t+1],[j]}),\rk(M_{[u_t-1],[j]})+(i-u_t+1)\}.\]

Note that $u_t-1,v_t+1\notin I$. By the definition of $X_A^I$ and Lemma~\ref{lem:ASMI}, \[\rk(M_{[i],[j]})\leq\min\{r_A(v_t+1,j),r_A(u_t-1,j)+(i-u_t+1)\}=r_A(i,j).\]So $M\in X_A$ and $X_A^I\subset X_A$ as desired.
\end{proof}
To prove Proposition~\ref{prop:variety}, another ingredient is the following lemma.
\begin{lemma}\label{lem:capofasmi}
Let $A_1,\cdots,A_k,A\in\ASM(n)$, and $X_{A_1}\cup\cdots\cup X_{A_k}=X_{A}$.
\begin{enumerate}
    \item For each $j\in[n]$, there exists $t\in[k]$, such that $r_A(i,j)=r_{A_t}(i,j)$ for all $i\in[n]$.
    \item If $A_1,\cdots,A_k\in\ASM^I(n)$, then $A\in\ASM^I(n)$.
\end{enumerate}
\end{lemma}
\begin{proof}[Proof of Theorem~\ref{prop:variety}]
(2) is a corollary of (1) combined with Lemma \ref{lem:ASMI}. So we only need to prove (1).

Assume for the sake of contradiction that there exists $b\in[n]$, such that for all $t\in[k]$, there exists $i\in[n]$ satisfying $r_A(i,b)\neq r_{A_t}(i,b)$. Now we fix such $b$ and construct a permutation $w\in S_n$, such that $w\in X_A$ and $r_w(i,b)=r_A(i,b)$ for all $i\in[n]$.

Since $\partial_A(i,b):=r_A(i,b)-r_A(i-1,b)\in\{0,1\}$ for all $i\in[n]$, and $b=r_A(n,b)=\sum_{i=1}^n\partial_A(i,b)$, the number of $\partial_A(i,b)$ that equals to $1$ is $b$. Suppose $\partial_A(i,b)=1$ for $i\in U:=\{u_1>u_2>\cdots>u_b\}$ and $\partial_A(i,b)=0$ for $i\in V:=\{v_{b+1}>v_{b+2}>\cdots>v_n\}$. Define $w\in S_n$ by $w(u_i)=i$ for $1\leq i\leq b$ and $w(v_i)=i$ for $b<i\leq n$. We check that
\begin{align*}
r_w(u_t,b)&=|\{i\:|\: i\leq u_t,w(i)\leq b\}|=|\{l\:|\:u_l\leq u_t\}|\\
&=\sum_{u_l\leq u_t}\partial_A(u_l,b)=\sum_{0\leq i\leq u_t}\partial_A(i,b)=r_A(u_t,b),\quad\text{and that}\\
r_w(v_t,b)&=v_t-|\{i\:|\:i\leq v_t,w(i)>b\}|=v_t-|\{l\:|\:v_l\leq v_t\}|=v_t-\sum_{v_l\leq v_t}(1-\partial(v_l,b))\\
&=\sum_{1\leq i\leq v_t}(1-\partial_A(i,b))=\sum_{1\leq i\leq v_t}\partial_A(i,b)=r_A(v_t,b).
\end{align*}
So $r_w(i,b)=r_A(i,b)$ for all $i\in[n]$. Then we show that $w\in X_A$. We claim that\[r_w(i,j)=\begin{cases}
\max\{0,r_w(i,b)-b+j\},&\text{ for }j\leq b,\\
\max\{r_w(i,b),i+j-n\},&\text{ for }j\geq b.
\end{cases}\]

If $j\leq b$, we may assume $i\in U$, because for $i\in V$, we can choose the maximal $u_t<i$ and then $r_w(i,j)=r_w(u_t,j)$ for all $j\leq b$. If there is no such $u_t<i$, we can choose the smallest $u_t>i$ and argue in the same way. Now set $i=u_t$ and $j\leq b$. We have
\begin{align*}
r_w(u_t,j)&=|\{i\:|\:i\leq u_t,w(i)\leq j\}|=|\{l\:|\:t\leq l\leq b,w(u_l)\leq j\}|\\
&=|\{l\:|\:t\leq l\leq j\}|=\max\{j-t+1,0\}.  
\end{align*}It satisfies $r_w(u_t,j)=\max\{0,r_w(u_t,b)-b+j\}$.

If $j\geq b$, we may assume $i\in V$, because for $i\in U$, we can choose the maximal $v_t<i$ and then $r_w(i,j)=r_w(v_t,j)-v_t+j$ for all $j\geq b$. If there is no such $v_t<i$, we can choose the smallest $v_t>i$ and argue in the same way. Now set $i=v_t$ and $j\geq b$. We have
\begin{align*}
r_w(v_t,j)&=v_t-|\{i\:|\:i\leq v_t,w(i)>j\}|=v_t-|\{l\:|\:t\leq l\leq n,w(v_l)>j\}|\\
&=v_t-|\{l\:|\:\max\{t,j+1\}\leq l\leq n\}|=v_t-n+\max\{t-1,j\}.
\end{align*}
It satisfies $r_w(v_t,j)=\max\{r_w(v_t,b),v_t+j-n\}$.

Finally, since $r_A$ is a corner sum matrix, 
\[r_A(i,j)\geq\begin{cases}
\max\{0,r_A(i,b)-b+j\},&\text{ for }j\leq b,\\
\max\{r_A(i,b),i+j-n\},&\text{ for }j\geq b.
\end{cases} \]
So if $j\leq b$,\[r_w(i,j)=\max\{0,r_w(i,b)-b+j\}=\max\{0,r_A(i,b)-b+j\}\leq r_A(i,j),\]and if $j\geq b$,\[r_w(i,j)=\max\{r_w(i,b),i+j-n\}=\max\{r_A(i,b),i+j-n\}\leq r_A(i,j).\]
Therefore, $r_w(i,j)\leq r_A(i,j)$ for all $i,j\in[n]$ meaning that $w\in X_A$.

By the assumption of contradiction, for each $t\in[k]$, there exists $i\in[n]$ such that $r_A(i,b)\neq r_{A_t}(i,b)$. Since $X_A\supset X_{A_t}$, we must have $r_A(i,b)>r_{A_t}(i,b)$. Combined with the fact that $r_A(i,b)=r_w(i,b)$, we get $w\notin X_{A_t}$. So $w\in X\backslash\cup_{t=1}^kX_{A_t}$, contradicting the fact that $X_{A_1}\cup\cdots\cup X_{A_k}=X_{A}$.
\end{proof}
We are now ready to prove Proposition~\ref{prop:variety}.
\begin{proof}[Proof of Proposition~\ref{prop:variety}]
(1) follows from Theorem~\ref{thm:asmvar} and Theorem~\ref{thm:meet-operation}. For (2), we prove the statement about the varieties, and then the statement about the ideals follows.

By Lemma~\ref{lem:asmvar}, $A\leq_I B$ if and only if $X_A\supset X_B$, where $A,B\in\ASM^I(n)$. Since $A=A_1\wedge^I A_2\wedge^I\cdots\wedge^I A_k$, we have $A\leq_I A_t$ for all $t\in[k]$. Thus, $X_A\supset X_{A_t}$, and then $X_A\supset\cup_{t=1}^k X_{A_t}$. Since $X_{A_1}\cup\cdots\cup X_{A_k}$ is an ASM variety, let it be $X_{U}$. By Lemma~\ref{lem:capofasmi}, $U\in\ASM^I(n)$. For all $t\in[k]$, $X_{A_t}\subset X_U$ implies $A_t\geq U$. Since $A$ is the meet of $A_t$, we have $A\geq U$. Therefore, $X_A\subset X_U$ and $U=A$. 
\end{proof}
\section{Special cases and further discussions}\label{sec:further}
\subsection{$|I|=n-2$ is maximal}
When $|I| = n-2$, $S_n^I$ is the order ideal of Young's lattice under a rectangle, and is thus already a lattice. As a sanity check to our main theorem, we use Theorem~\ref{thm:main} to explain that $\ASM^I(n)=S_n^I$ in this case.
\begin{lemma}\label{lem:In-2}
$s,t\in[n-1]$. Suppose that $I_1=\{1,2,\cdots,s\}\subset I$ and $I_2=\{t,t+1,\cdots,n-1\} \subset I$. Let $A\in\ASM^I(n)$. Then $A_{i,j}\neq-1$ if $i\in I_1\cup I_2\cup\{s+1,n\}$.
\end{lemma}
\begin{proof}
Define $I_1^+=I_1\cup\{s+1\},I_2^+=I_2\cup\{n\}$. For each $i\in I_1^+$, we define $j_i$ to be the smallest integer such that $A_{i,j_i}=1$. We claim that $j_i\leq j_{i+1}$ for all $i\in I_1$. Otherwise, let $i\in I_1$ be such that $j_i>j_{i+1}$. Then $r_A(i+1,j_{i+1})=r_A(i,j_{i+1})+1$ but $r_A(i,j_{i+1})=r_A(i-1,j_{i+1})$, contradicting the definition of $\ASM^I(n)$.

Now assume for the sake of contradiction that $A_{i,j}=-1$ for some $i\in I_1^+$. Choose $A_{u,v} = -1$ with the smallest row index $u$. Then there exists $x<u$, such that $A_{x,v} =1$. But we also have $A_{x,j_x} = 1$ where $j_x\leq j_u <v$, leading to the fact that $A_{x,y}=-1$ for some $y\in\{j_x+1,j_x+2,\cdots,v-1\}$, contradicting the minimality of $u$.

By setting $k_i$ to be the largest integer such that $A_{i,k_i} = 1$, we can prove that $A_{i,j}\neq-1$ for $i\in I_2^+$ in the same way.
\end{proof}
\begin{proposition}\label{prop:I-maximal}
Suppose that $I\subset[n-1]$ has cardinality $n-2$, then $\ASM^I(n)=S_n^I$.
\end{proposition}
\begin{proof}
Suppose $[n-1]=I\sqcup\{a\}$. Given $A\in\ASM^I(n)$, we set $I_1=\{1,2,\cdots,a-1\}$ and $I_2=\{a+1,a+2,\cdots,n-1\}$ in Lemma \ref{lem:In-2}. Then $A_{i,j}\neq-1$ for all $i\in[n]$. So $A\in S_n^I$.
\end{proof}
\subsection{$I=\{t,t+1\cdots,n-1\}$}
\begin{proof}[Proof of Proposition~\ref{prop:J=1...n-t}]
Readers are refereed to Section~\ref{sec:prelim} for the bijection $\varphi:\ASM(n)\rightarrow \MT(1,2,\ldots,n)$ from alternating sign matrices of size $n$ to the set of monotone triangles with bottom row $(1,2,\ldots,n)$.

Given $A\in\ASM^I(n)$, the first $(t-1)$ rows of $\varphi(A)$ form a monotone triangle of size $(t-1)$ where the entries are less than or equal to $n$. For $A\in\ASM^I(n)$, $A_{i,j}\neq-1$ for all $i\in\{t,t+1,\cdots,n-1\}$ by Lemma~\ref{lem:ASMI}. Thus, there is exactly one $1$ in each of these rows. Let $A_{i,j_i} = 1$ for $i\in\{t,t+1,\cdots,n-1\}$. By the definition of $\ASM^I(n)$, the sequence $\{j_t,j_{t+1},\cdots,j_{n-1}\}$ must be strictly increasing. So $A$ is determined by its first $(t-1)$ rows. Thus, the map $\varphi$ restricted to the first $(t-1)$ rows is injective. This map is also surjective since whole procedure can be reversed. An example is shown in Figure~\ref{fig:bijection-ASMI-MT-example}.
\end{proof}
\begin{figure}[h!]
\centering
\[\begin{pmatrix}
    0&0&0&1&0&0\\
    0&0&1&-1&1&0\\
    1&0&-1&1&-1&1\\
    0&1&0&0&0&0\\
    0&0&1&0&0&0\\
    0&0&0&0&1&0
\end{pmatrix}\qquad\begin{array}{c}
 4 \\
 3 \quad 5 \\
 1 \quad 4 \quad 6\end{array}\]
\caption{An example of the bijection between $\ASM^{\{t,t+1,\ldots,n-1\}}(n)$ (left) and monotone triangles with entries $\leq n$ (right), with $t=4$ and $n=6$.}
\label{fig:bijection-ASMI-MT-example}
\end{figure}

We can use the following theorem to provide exact formulas for $|\ASM^{\{t,t+1,\ldots,n-1\}}(n)|$ for small values of $t$, while an exact formula for general $t$ seems out of reach.
\begin{theorem}[\cite{Fis06}]\label{thm:MT_num}
    The number of monotone triangles with $m$ rows and prescribed bottom row $(k_1,k_2,\cdots,k_m)$ is given by
    \begin{align*}
        \left(\prod_{1\leq p<q\leq n}\left(\mathrm{id}+E_{k_p}\Delta_{k_q}\right)\right)\prod_{1\leq i<j\leq n}\dfrac{k_j-k_i}{j-i}
    \end{align*}
    where $E_x$ denotes the shift operator, defined by $E_xp(x) = p(x+1)$, and $\Delta_x = E_x - \mathrm{id}$ denotes the difference operator.
\end{theorem}
\begin{theorem}\label{cor:n-t}
We have the following enumeration results:
\begin{align*}
|\ASM^{\{3,4,\cdots,n-1\}}(n)|&=\dfrac{(n-1)n(n+4)}{6},\quad\text{and}\\
|\ASM^{\{4,5,\cdots,n-1\}}(n)|&=\dfrac{(n-2)(n-1)n(n+1)(n^2+14n+54)}{360}.
\end{align*}
In general, for fixed $t$, $|\ASM^{\{t,t+1,\cdots,n-1\}}(n)|$ is a polynomial in $n$ with degree $\frac{t(t-1)}{2}$, and leading coefficient \[\prod_{j=1}^{t-2}\frac{j!}{(2j+1)!}.\]
\end{theorem}
\begin{proof}
Denote the number of monotone triangles with $m$ rows and prescribed bottom row $(k_1,k_2,\cdots,k_m)$ by $\alpha(k_1,k_2,\cdots,k_m)$. Then $\alpha(k_1,k_2)=1-k_1+k_2$. We calculate that 
\begin{align*}
|\ASM^{\{3,4,\cdots,n-1\}}(n)|&=\sum_{1\leq k_1<k_2\leq n} (1-k_1+k_2)\\
&=\sum_{1\leq t\leq n-1}\sum_{1\leq k_1\leq n-t}(1-k_1+(k_1+t))\\
&=\sum_{1\leq t\leq n-1}(n-t)(1+t)\\
&=\dfrac{(n-1)n(n+4)}{6}.
\end{align*}
The formula for $|\ASM^{\{4,5,\ldots,n-1\}}(n)|$ is computed in the same way. Using the calculation result in \cite{Fis06}, $\alpha(k_1,k_2,k_3)=\frac{1}{2}(-3k_1+k_1^2+2k_1k_2-k_1^2k_2-2k_2^2+k_1k_2^2+3k_3-4k_1k_3+k_1^2k_3+2k_2k_3-k_2^2k_3+k_3^2-k_1k_3^2+k_2k_3^2)$. Skipping technical calculations, we have
\begin{align*}
|\ASM^{\{4,5,\cdots,n-1\}}(n)|&=\sum_{1\leq k_1<k_2<k_3\leq n}\alpha(k_1,k_2,k_3)\\
&=\dfrac{(n-2)(n-1)n(n+1)(n^2+14n+54)}{360}.    
\end{align*}
By Theorem \ref{thm:MT_num}, the highest degree part of $\alpha(k_1,k_2,\cdots,k_{t-1})$ is $\prod_{1\leq i<j\leq t-1}\frac{k_j-k_i}{j-i}$, so the highest degree part of $|\ASM^{\{t,t+1,\cdots,n-1\}}(n)|$ is \[\sum_{1\leq k_1<\cdots<k_{t-1}\leq n} \prod_{1\leq i<j\leq t-1}\dfrac{k_j-k_i}{j-i}=n^{\frac{(t-1)(t-2)}{2}}\sum_{\substack{0<k_1<\cdots<k_{t-1}\leq 1,\\k_i\in\{\frac{1}{n},\cdots,\frac{n-1}{n},1\}}}\prod_{1\leq i<j\leq t-1}\dfrac{k_j-k_i}{j-i}.\] As $n\rightarrow\infty$, it is asymptotic to the following term, which can be calculated using Selburg integral (see for example \cite{Sel44}),
\begin{align*}
&\quad\;\;n^{\frac{(t-1)(t-2)}{2}}\cdot n^{t-1}\int_{0<x_1<\cdots<x_{t-1}\leq1}\prod_{1\leq i<j\leq t-1}\dfrac{x_j-x_i}{j-i}\mathrm{d}x_1\cdots\mathrm{d}x_{t-1}\\
&=\dfrac{n^{\frac{t(t-1)}{2}}}{(t-1)!}\prod_{1\leq i<j\leq t-1}\frac{1}{j-i}\cdot\int_{0<x_1,\cdots,x_{t-1}\leq1}\prod_{1\leq i<j\leq t-1}|x_j-x_i|\mathrm{d}x_1\cdots\mathrm{d}x_{t-1}\\
&=\dfrac{n^{\frac{t(t-1)}{2}}}{(t-1)!}\prod_{j=1}^{t-2}\frac{1}{j!}\cdot \prod_{j=0}^{t-2}\dfrac{\Gamma(1+\frac{j}{2})^2\Gamma(1+\frac{j+1}{2})}{\Gamma(2+\frac{t+j-2}{2})\Gamma(\frac{3}{2})}\\
&=\dfrac{n^{\frac{t(t-1)}{2}}}{(t-1)!}\prod_{j=0}^{t-2}\frac{1}{j!}\cdot\prod_{j=0}^{t-2}\dfrac{2^{-j-1}\sqrt{\pi}(j+1)!\Gamma(1+\frac{j}{2})}{\frac{\sqrt{\pi}}{2}\Gamma(\frac{t+j+2}{2})}\\
&=n^{\frac{t(t-1)}{2}}2^{-\frac{(t-1)(t-2)}{2}} R_t,
\end{align*}
where $R_t=\prod_{j=0}^{t-2}\left(\Gamma(\frac{j+2}{2})/\Gamma(\frac{t+j+2}{2})\right).$  We have \[\frac{R_t}{R_{t-1}}=\frac{\Gamma(\frac{t}{2})\Gamma(\frac{t+1}{2})}{\Gamma(t)\Gamma(\frac{2t-1}{2})}=\frac{2^{1-t}\sqrt{\pi}}{\Gamma(t-\frac{1}{2})}=\frac{2^{t-2}(t-2)!}{(2t-3)!}.\] Thus, \[R_t=R_2\prod_{j=3}^{t}\dfrac{R_j}{R_{j-1}}=2^{\frac{(t-1)(t-2)}{2}}\prod_{j=1}^{t-2}\frac{j!}{(2j+1)!}.\] We conclude that as $n$ tends to infinity, \[|\ASM^{\{t,t+1,\cdots,n-1\}}|\sim \prod_{j=1}^{t-2}\frac{j!}{(2j+1)!}\cdot n^{\frac{t(t-1)}{2}}.\]
\end{proof}
\subsection{Further discussion}
There are other technical tools and languages that can be helpful towards showing that the MacNeille completion of $S_n^I$ is exactly $\ASM^I(n)$.
\begin{definition}[\cite{Rea02}]
An equivalence relation $\Theta$ on a poset $P$ is called a \emph{congruence} if
\begin{itemize}
\item every equivalence class is an interval;
\item the map that sends each $a\in P$ to the minimal element in $[a]_\Theta$ is order-preserving;
\item the map that sends each $a\in P$ to the maximal element in $[a]_\Theta$ is order-preserving.
\end{itemize} 
\end{definition}
\begin{theorem}[\cite{Rea02}]\label{thm:reading}
    Let $P$ be a finite poset with MacNeille completion $L(P)$, and let $\Theta$ be an equivalence relation on $P$. Then $\Theta$ is a congruence on $P$ if and only if there is a congruence $L(\Theta)$ on $L(P)$ which restricts exactly to $\Theta$, in which case
    \begin{itemize}
        \item $L(\Theta)$ is the unique congruence on $L(P)$ which restricts exactly to $\Theta$, and
        \item the MacNeille completion $L(P/\Theta)$ is naturally isomorphic to $L(P) / L(\Theta)$.
    \end{itemize}
\end{theorem}
Now let $P$ be the Bruhat order on $S_n$ and $\Theta$ be the equivalence relation that $[w]_\Theta = [v]_\Theta$ if $wS_n(I)=vS_n(I)$. It can be shown that $L(\Theta)$ is exactly the relation $=_I$ in Definition \ref{def:leqI}. Theorem~\ref{thm:reading} now states that the MacNeille completion of $S_n^I$ is $\ASM(n)/=_I$, which can be shown to be isomorphic to $\ASM^I(n)$ after some technical arguments.

Another tool is the operator $\pi_i$, which was first defined on monotone triangles in \cite{HR20}, and presented in the language of ASMs in \cite{EKW25}.
\begin{definition}[\cite{EKW25}]
Given $A\in\ASM(n)$ and $i\in[n-1]$, define \[\pi_i(A):=\min\{B\in\ASM(n)\mid r_A(a,b)=r_B(a,b)\text{ for all }a,b\in [n]\text{ with }a\neq i\}.\]
\end{definition}
Since these operators satisfy the commutation relations $\pi_i\circ\pi_j=\pi_j\circ\pi_i$ for $|i-j| > 1$, and the braid relations $\pi_i\circ\pi_{i+1}\circ\pi_i=\pi_{i+1}\circ\pi_i\circ\pi_{i+1}$ for $i\in[n-2]$, one can define $\pi_w:=\pi_{i_1}\circ\cdots\circ\pi_{i_{\ell}}$ for arbitrary $w\in S_n$ with a reduced expression $w=s_{i_1}\cdots s_{i_{\ell}}$.

Given $I\subset[n-1]$, with some technical arguments, we can prove that
\[\pi_I(A)=\min\{B\in\ASM(n)\mid r_A(a,b)=r_B(a,b)\text{ for all }a,b\in[n]\text{ with }a\notin I\},\]
where $\pi_I:=\pi_{w_0^I}$ and $w_0^I$ is the maximal element in $S_n^I$. This is a slight generalization of the $\pi_i$'s in \cite{EKW25}.
One can show that $\pi_I(\ASM(n))=\ASM^I(n)$.

\begin{proposition}
For $w\in S_n$, $\pi_w(\ASM(n))$ is a lattice.
\end{proposition}
\begin{proof}
Denote $\pi_w(\ASM(n))$ by $\Pi_w$ for convenience. We prove the following statement inductively: $\Pi_w$ is a lattice, and the join operation on it coincides with that of $\ASM(n)$. 

If $w=\mathrm{id}$, $\Pi_{w}=\ASM(n)$ is a lattice. Then we assume $\Pi_w$ is a lattice for all $w\in S_n$ with $\ell(w)\leq m$. Given $w\in S_n$ with $\ell(w)=m+1$, write $w=s_iv$, where $\ell(v)=\ell(w)-1$. Then $\Pi_w=\pi_i(\Pi_v)$. Given arbitrary $A,B\in\Pi_v$, we will show that the join of $\pi_i(A)$ and $\pi_i(B)$ in $\Pi_w$ is $\pi_i(A)\vee\pi_i(B)$. By \cite[Proposition~3.4]{EKW25}, $\pi_i(A)\vee\pi_i(B)=\pi_i(A\vee B)$, and by the induction hypothesis, $A\vee B\in\Pi_v$. So $\pi_i(A)\vee\pi_i(B)\in\Pi_w$, and it is the join of $\pi_i(A)$ and $\pi_i(B)$ in $\Pi_w$. The minimal element of $\Pi_w$ is the identity matrix. By Proposition \ref{lattice}, $\Pi_w$ is a lattice.
\end{proof}
\begin{question}
Is $\pi_w(\ASM(n))$ the MacNeille completion of some subposet of $S_n$?
\end{question}
$\ASM^I(n)$ can also be nicely described using the language of six-vertex models. The six-vertex model is a square-lattice model in which each vertex satisfies the ice rule, namely, exactly two arrows point into the vertex and two arrows point out. A particular specialization is the model with domain wall boundary conditions, where arrows point inward at the left and right sides, and outward at the top and bottom. A state of the six-vertex model is an assignment of orientations to all edges of the lattice such that the ice rule at each vertex and the domain wall boundary conditions are satisfied. For a square-lattice with $n\times n$ vertices, we denote the set of its states by $\mathrm{St}(n)$. Given a state, we can convert it into an ASM by the following correspondence:

\begin{tikzpicture}[scale=0.92]

\begin{scope}[xshift=0cm]
  \draw[midarrow=2.2mm] (0,0)--(0,1);      
  \draw[midarrow=2.2mm] (0,0)--(0,-1);     
  \draw[midarrow=2.2mm] (-1,0)--(0,0);     
  \draw[midarrow=2.2mm] (1,0)--(0,0);      
  \node at (0,-1.6) {$1$};
\end{scope}

\begin{scope}[xshift=3cm]
  \draw[midarrow=2.2mm] (0,1)--(0,0);      
  \draw[midarrow=2.2mm] (0,-1)--(0,0);     
  \draw[midarrow=2.2mm] (0,0)--(-1,0);     
  \draw[midarrow=2.2mm] (0,0)--(1,0);      
  \node at (0,-1.6) {$-1$};
\end{scope}

\begin{scope}[xshift=6cm]
  \draw[midarrow=2.2mm] (0,0)--(0,1);      
  \draw[midarrow=2.2mm] (0,-1)--(0,0);     
  \draw[midarrow=2.2mm] (-1,0)--(0,0);     
  \draw[midarrow=2.2mm] (0,0)--(1,0);      
  \node at (0,-1.6) {$0$};
\end{scope}

\begin{scope}[xshift=9cm]
  \draw[midarrow=2.2mm] (0,1)--(0,0);      
  \draw[midarrow=2.2mm] (0,0)--(0,-1);     
  \draw[midarrow=2.2mm] (0,0)--(-1,0);     
  \draw[midarrow=2.2mm] (1,0)--(0,0);      
  \node at (0,-1.6) {$0$};
\end{scope}

\begin{scope}[xshift=12cm]
  \draw[midarrow=2.2mm] (0,0)--(0,1);      
  \draw[midarrow=2.2mm] (0,-1)--(0,0);     
  \draw[midarrow=2.2mm] (0,0)--(-1,0);     
  \draw[midarrow=2.2mm] (1,0)--(0,0);      
  \node at (0,-1.6) {$0$};
\end{scope}

\begin{scope}[xshift=15cm]
  \draw[midarrow=2.2mm] (0,1)--(0,0);      
  \draw[midarrow=2.2mm] (0,0)--(0,-1);     
  \draw[midarrow=2.2mm] (-1,0)--(0,0);     
  \draw[midarrow=2.2mm] (0,0)--(1,0);      
  \node at (0,-1.6) {$0$};
\end{scope}

\end{tikzpicture}

See \cite{MR1226347,MR865837} for more details. For $I\subset [n-1]$, we define $\mathrm{St}_I(n)\subset \mathrm{St}(n)$ to be the states such that if an arrow points to the right in the $i$-th row, where $i\in I$, then the arrow directly below it in the $(i+1)$-th row must also point to the right. Then the states that correspond to matrices in $\ASM^I(n)$ are exactly the states in $\mathrm{St}_I(n)$. See Figure \ref{fig:bijection-ASMI-SIXM-example} for an example, where $n=6$ and $I=\{3\}$.
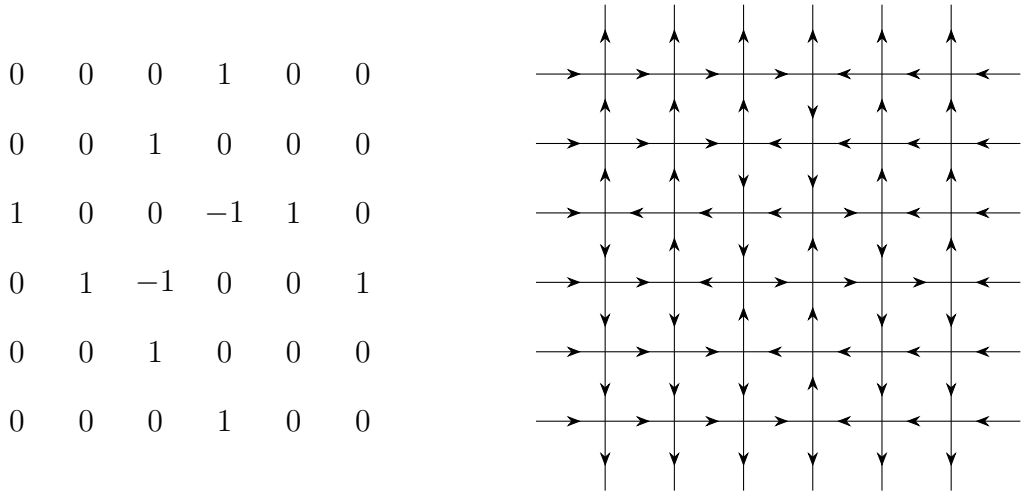
\begin{figure}[h!]
\centering
\begin{gather*}
\begin{tikzpicture}[scale=0.92]
\begin{scope}[xshift=3cm]
\node at (-11,0) {$0$};
\node at (-10,0) {$0$};
\node at (-9,0) {$0$};
\node at (-8,0) {$1$};
\node at (-7,0) {$0$};
\node at (-6,0) {$0$};
\node at (-11,-1) {$0$};
\node at (-10,-1) {$0$};
\node at (-9,-1) {$1$};
\node at (-8,-1) {$0$};
\node at (-7,-1) {$0$};
\node at (-6,-1) {$0$};
\node at (-11,-2) {$1$};
\node at (-10,-2) {$0$};
\node at (-9,-2) {$0$};
\node at (-8,-2) {$-1$};
\node at (-7,-2) {$1$};
\node at (-6,-2) {$0$};
\node at (-11,-3) {$0$};
\node at (-10,-3) {$1$};
\node at (-9,-3) {$-1$};
\node at (-8,-3) {$0$};
\node at (-7,-3) {$0$};
\node at (-6,-3) {$1$};
\node at (-11,-4) {$0$};
\node at (-10,-4) {$0$};
\node at (-9,-4) {$1$};
\node at (-8,-4) {$0$};
\node at (-7,-4) {$0$};
\node at (-6,-4) {$0$};
\node at (-11,-5) {$0$};
\node at (-10,-5) {$0$};
\node at (-9,-5) {$0$};
\node at (-8,-5) {$1$};
\node at (-7,-5) {$0$};
\node at (-6,-5) {$0$};
\draw[midarrow=1.9mm] (-3.5,0)--(-2.5,0);  
\draw[midarrow=1.9mm] (-2.5,0)--(-1.5,0);      
\draw[midarrow=1.9mm] (-1.5,0)--(-0.5,0);     
\draw[midarrow=1.9mm] (-0.5,0)--(0.5,0);    
\draw[midarrow=1.9mm] (1.5,0)--(0.5,0);     
\draw[midarrow=1.9mm] (2.5,0)--(1.5,0);    
\draw[midarrow=1.9mm] (3.5,0)--(2.5,0);
\draw[midarrow=1.9mm] (-3.5,-1)--(-2.5,-1);  
\draw[midarrow=1.9mm] (-2.5,-1)--(-1.5,-1);      
\draw[midarrow=1.9mm] (-1.5,-1)--(-0.5,-1);     
\draw[midarrow=1.9mm] (0.5,-1)--(-0.5,-1);    
\draw[midarrow=1.9mm] (1.5,-1)--(0.5,-1);     
\draw[midarrow=1.9mm] (2.5,-1)--(1.5,-1);    
\draw[midarrow=1.9mm] (3.5,-1)--(2.5,-1);
\draw[midarrow=1.9mm] (-3.5,-2)--(-2.5,-2);  
\draw[midarrow=1.9mm] (-1.5,-2)--(-2.5,-2);      
\draw[midarrow=1.9mm] (-0.5,-2)--(-1.5,-2);     
\draw[midarrow=1.9mm] (0.5,-2)--(-0.5,-2);    
\draw[midarrow=1.9mm] (0.5,-2)--(1.5,-2);     
\draw[midarrow=1.9mm] (2.5,-2)--(1.5,-2);    
\draw[midarrow=1.9mm] (3.5,-2)--(2.5,-2);  
\draw[midarrow=1.9mm] (-3.5,-3)--(-2.5,-3);  
\draw[midarrow=1.9mm] (-2.5,-3)--(-1.5,-3);      
\draw[midarrow=1.9mm] (-0.5,-3)--(-1.5,-3);     
\draw[midarrow=1.9mm] (-0.5,-3)--(0.5,-3);    
\draw[midarrow=1.9mm] (0.5,-3)--(1.5,-3);     
\draw[midarrow=1.9mm] (1.5,-3)--(2.5,-3);    
\draw[midarrow=1.9mm] (3.5,-3)--(2.5,-3);  
\draw[midarrow=1.9mm] (-3.5,-4)--(-2.5,-4);  
\draw[midarrow=1.9mm] (-2.5,-4)--(-1.5,-4);      
\draw[midarrow=1.9mm] (-1.5,-4)--(-0.5,-4);     
\draw[midarrow=1.9mm] (0.5,-4)--(-0.5,-4);    
\draw[midarrow=1.9mm] (1.5,-4)--(0.5,-4);     
\draw[midarrow=1.9mm] (2.5,-4)--(1.5,-4);    
\draw[midarrow=1.9mm] (3.5,-4)--(2.5,-4);  
\draw[midarrow=1.9mm] (-3.5,-5)--(-2.5,-5);  
\draw[midarrow=1.9mm] (-2.5,-5)--(-1.5,-5);      
\draw[midarrow=1.9mm] (-1.5,-5)--(-0.5,-5);     
\draw[midarrow=1.9mm] (-0.5,-5)--(0.5,-5);    
\draw[midarrow=1.9mm] (1.5,-5)--(0.5,-5);     
\draw[midarrow=1.9mm] (2.5,-5)--(1.5,-5);    
\draw[midarrow=1.9mm] (3.5,-5)--(2.5,-5);
\draw[midarrow=1.9mm] (-2.5,0)--(-2.5,1);  
\draw[midarrow=1.9mm] (-2.5,-1)--(-2.5,0);      
\draw[midarrow=1.9mm] (-2.5,-2)--(-2.5,-1);     
\draw[midarrow=1.9mm] (-2.5,-2)--(-2.5,-3);    
\draw[midarrow=1.9mm] (-2.5,-3)--(-2.5,-4);     
\draw[midarrow=1.9mm] (-2.5,-4)--(-2.5,-5);    
\draw[midarrow=1.9mm] (-2.5,-5)--(-2.5,-6);
\draw[midarrow=1.9mm] (-1.5,0)--(-1.5,1);  
\draw[midarrow=1.9mm] (-1.5,-1)--(-1.5,0);      
\draw[midarrow=1.9mm] (-1.5,-2)--(-1.5,-1);     
\draw[midarrow=1.9mm] (-1.5,-3)--(-1.5,-2);    
\draw[midarrow=1.9mm] (-1.5,-3)--(-1.5,-4);     
\draw[midarrow=1.9mm] (-1.5,-4)--(-1.5,-5);    
\draw[midarrow=1.9mm] (-1.5,-5)--(-1.5,-6);
\draw[midarrow=1.9mm] (-0.5,0)--(-0.5,1);  
\draw[midarrow=1.9mm] (-0.5,-1)--(-0.5,0);      
\draw[midarrow=1.9mm] (-0.5,-1)--(-0.5,-2);     
\draw[midarrow=1.9mm] (-0.5,-2)--(-0.5,-3);    
\draw[midarrow=1.9mm] (-0.5,-4)--(-0.5,-3);     
\draw[midarrow=1.9mm] (-0.5,-4)--(-0.5,-5);    
\draw[midarrow=1.9mm] (-0.5,-5)--(-0.5,-6);
\draw[midarrow=1.9mm] (0.5,0)--(0.5,1);  
\draw[midarrow=1.9mm] (0.5,0)--(0.5,-1);      
\draw[midarrow=1.9mm] (0.5,-1)--(0.5,-2);     
\draw[midarrow=1.9mm] (0.5,-3)--(0.5,-2);    
\draw[midarrow=1.9mm] (0.5,-4)--(0.5,-3);     
\draw[midarrow=1.9mm] (0.5,-5)--(0.5,-4);    
\draw[midarrow=1.9mm] (0.5,-5)--(0.5,-6);
\draw[midarrow=1.9mm] (1.5,0)--(1.5,1);  
\draw[midarrow=1.9mm] (1.5,-1)--(1.5,0);      
\draw[midarrow=1.9mm] (1.5,-2)--(1.5,-1);     
\draw[midarrow=1.9mm] (1.5,-2)--(1.5,-3);    
\draw[midarrow=1.9mm] (1.5,-3)--(1.5,-4);     
\draw[midarrow=1.9mm] (1.5,-4)--(1.5,-5);    
\draw[midarrow=1.9mm] (1.5,-5)--(1.5,-6);
\draw[midarrow=1.9mm] (2.5,0)--(2.5,1);  
\draw[midarrow=1.9mm] (2.5,-1)--(2.5,0);      
\draw[midarrow=1.9mm] (2.5,-2)--(2.5,-1);     
\draw[midarrow=1.9mm] (2.5,-3)--(2.5,-2);    
\draw[midarrow=1.9mm] (2.5,-3)--(2.5,-4);     
\draw[midarrow=1.9mm] (2.5,-4)--(2.5,-5);    
\draw[midarrow=1.9mm] (2.5,-5)--(2.5,-6);
\end{scope}
\end{tikzpicture}
\end{gather*}
\caption{An example of the bijection between $\ASM^{I}(n)$ (left) and $\mathrm{St}_I(n)$ (right), with $n=6$ and $I=\{3\}$.}
\label{fig:bijection-ASMI-SIXM-example}
\end{figure}

\section*{Acknowledgement}
We thank Anna Weigandt for helpful conversations and for pointing us towards useful references. 
\bibliographystyle{alpha}
\bibliography{ref}

\end{document}